\documentclass[a4paper,11pt]{article}

\usepackage[T1]{fontenc}
\usepackage[english]{babel}
\usepackage{dsfont}
\usepackage{amsfonts}
\usepackage{amsmath}
\usepackage{amssymb,latexsym}
\usepackage{mathrsfs}
\usepackage{enumitem}
\usepackage{ stmaryrd }
\usepackage{amsthm}
\usepackage{authblk}
\usepackage{csquotes}
\usepackage{subcaption}
\usepackage{hyperref}
\usepackage{cleveref}

\usepackage{graphicx} 
\graphicspath{ {./images/} }

\usepackage[left=2.5cm, right=2.5cm, top=3cm]{geometry}

\newtheorem{lem}{Lemma}[section]
\newtheorem{prop}[lem]{Proposition}
\newtheorem{rem}[lem]{Remark}
\newtheorem{defi}[lem]{Definition}
\newtheorem{theo}[lem]{Theorem}

\renewcommand{\P}{\mathbb{P}}
\newcommand{\N}{\mathbb{N}}
\newcommand{\E}{\mathbb{E}}
\newcommand{\Z}{\mathbb{Z}}
\newcommand{\R}{\mathbb{R}}

\title{\textsc{\Huge{Brochette First-Passage Percolation}}}
\author{Maxime Marivain}
\affil{CMAP, \'Ecole Polytechnique, CNRS, Route de
Saclay, 91128 Palaiseau Cedex, France\\\emph{maxime.marivain@polytechnique.edu}}
\date{\today}
\begin{document}
\maketitle

\begin{abstract}
    We investigate a novel first-passage percolation model, referred to as the Brochette first-passage percolation model, where the passage times associated with edges lying on the same line are equal. First, we establish a point-to-point convergence theorem, identifying the time constant. In particular, we explore the case where the time constant vanishes and demonstrate the existence of a wide range of possible behaviours. Next, we prove a shape theorem, showing that the limiting shape is the $L^1$ diamond. Finally, we extend the analysis by proving a point-to-point convergence theorem in the setting where passage times are allowed to be infinite.
\end{abstract}
\emph{Keywords :} First-passage percolation, time constant, shape theorem.\\
\emph{AMS classification :} 60K35, 60C05, 60Fxx.

\section{Introduction}
\subsection{First-passage percolation in the independent setting}
 First-passage percolation is a probabilistic model introduced in 1965 by Hammersley and Welsh in \cite{Hammersley} to study the flow of fluids through porous media. We refer to \cite{Survey} for a survey on the subject.
 In the classical model, we take the set of vertices to be $\mathbb{Z}^{d}$. We consider the set of edges $\mathbb{E}^{d}$ with endpoints $x, y\in\mathbb{Z}^{d}$ such that $\lVert x-y\rVert_1=1$. For each edge $e\in\mathbb{E}^{d}$, we associate a non-negative random variable $\tau_{e}$ with common cumulative distribution function $F$, called the passage time of edge $e$. A path is a finite or infinite sequence of edges in $\mathbb{E}^{d}$ such that consecutive edges share a unique endpoint. 
 For such a path $\Gamma$, we define:
$$
T(\Gamma)=\sum_{e\in\Gamma}\tau_{e},
$$
which we call the passage time of $\Gamma$. To define the passage time between $x,y\in\mathbb{R}^{d}$, we first note $x'$ as the unique vertex in $\mathbb{Z}^{d}$ such that $x\in x'+[0;1)^{d}$. Now, if $x,y\in\mathbb{R}^{d}$, we define the passage time between $x$ and $y$ as:
\begin{equation}\label{temps}
T(x,y)=\inf_{\Gamma} T(\Gamma),
\end{equation}
where this infimum is taken over the set of paths $\Gamma$ with endpoints $x'$ and $y'$. 
We focus on the random balls for this pseudo-distance, fundamental objects of our study. For $t\geq 0$ let
\begin{equation}\label{boule}
 B_{t}=\{ x\in\mathbb{R}^{d}:T(0,x)\leq t\}
\end{equation}
 be the set of points reachable from the origin in a time less than $t$. This model has been extensively studied since its introduction in 1965 in the case where $(\tau_{e})_{e\in\E^{d}}$ are chosen to be independent and identically distributed. 
We now recall the relevant results in the classical case within the context of this article. A first result of this model is that $T(0,nx)$ grows linearly. More precisely  we have:
\begin{theo}[\cite{kingmanpap}, see Theorem 2.1 \cite{Survey}]\label{th:papclass}
Assume that $(\tau_{e})_{e\in\E^{d}}$ are i.i.d. random variables such that $\E[\min(\tau_{1},...,\tau_{2d})]<+\infty$, then for all $x\in\Z^d$ there exists a constant $\mu(x)\geq0$ (called the time constant) such that:
$$
\lim_{n\to+\infty}\frac{T(0,nx)}{n} =\mu(x)\; a.s. \text{ and in }L^{1}. 
$$
\end{theo}
However, in this model, no non-trivial distribution is known for which the time constant $\mu$ can be computed. In 1981, Cox and Durrett improved this result with their shape theorem:
\begin{theo}[\cite{Cox}]
    Assume now that $\E[\min(\tau_{1}^{d},...,\tau_{2d}^{d})]<+\infty$ and that $F(0)<p_{c}(d)$, where $p_{c}(d)$ is the critical parameter for bond percolation in $\Z^{d}$. Then, there exists a deterministic, compact and convex subset $B_{\ast}$ of $\R^{d}$ such that  for all $\epsilon>0$ we have:
    $$
    \P((1-\epsilon)B_{\ast}\subset \frac{B_{t}}{t}\subset (1+\epsilon)B_{\ast} \text{ for all large }t)=1.
    $$
    Furthermore, $B_{\ast}$ is symmetric with respect to the axes of \;$\mathbb{R}^{d}$ and has a non-empty interior.
\end{theo}
Several slightly different models have been introduced, we mention in particular
\begin{itemize}
    \item The case where no integrability condition holds for $\tau$, in particular where $\P(\tau=+\infty)>0$ has also been studied, notably by Garet and Marchand in \cite{garetmarchand} and by Cerf and Théret in \cite{Theret}. They also proved an analogue of Theorem \ref{th:papclass} with convergence in probability.
    \item The case where passage times are non-independent.
    For example, the stationary ergodic case has been studied in \cite{fppergo}. This model, still rich, has led to many analogous results to the classical ones.
\end{itemize}
 In contrast to this, we consider in this article a long-range dependence model, which, to our knowledge, has not yet been studied. The name of this  model is inspired by the one studied by Duminil-Copin, Hilario, Kozma and Sidoravicius in \cite{Brochette} in the context of bond percolation. As they do in their article we introduce a strong dependence along the lines.
 \subsection{Brochette first-passage percolation}
 As in the classical case we consider the graph $(\Z^{d},\E^{d})$. In this model, the passage times associated with edges on the same line are equal. However, the passage times associated with edges on distinct lines still are independent.
 Let us now define our model more formally.
 \begin{defi}
 An integer line is the set of points in $\Z^{d}$ that lie on a same line in $\R^{d}$ parallel to one of the axes. We denote by $\Delta$ the set of these lines.
 \end{defi}
 \begin{defi}[Brochette first-passage percolation]\label{hyp:model}
 Let $(\tau_{\delta})_{\delta\in\Delta}$ be a family of non-negative i.i.d random variables with common cumulative distribution function $F$ and indexed by the set of integer lines. In the Brochette first-passage percolation we associate the same passage time to all edges whose endpoints lie on the same integer line $\delta\in\Delta$. In other words, for an edge $e\in\E^d$, the passage time $\tau_e$ of $e$ is equal to $\tau_\delta$ where $\delta$ is the unique integer line to which $e$ belongs.
  \end{defi}
  From now on $T$ and $B_t$ denote passage times and random balls for Brochette first-passage percolation.
  Throughout the text, we assume that the support of our passage times is included in an interval of the form $[a;+\infty)$, where $a\geq 0$ and $a$ equals the infimum of this support. In particular we will use that :
 \begin{equation}\label{hypo:a}
     \forall \delta>0,\; F( a+\delta)>0 \;\;and\;\; F(a-\delta)=0.
 \end{equation}
In the Brochette model, $T(0,n)$ is more sensitive to small values of passage times. Small passage times generate kind of \emph{highways} that geodesics tend to go through.
Thus, the random balls do not look alike for the classical and the Brochette models. (See a simulation on Figure \ref{Figure 1}.)
 \begin{figure}
    \centering
    \begin{subfigure}{0.4\textwidth}
        \centering
        \includegraphics[width=\textwidth]{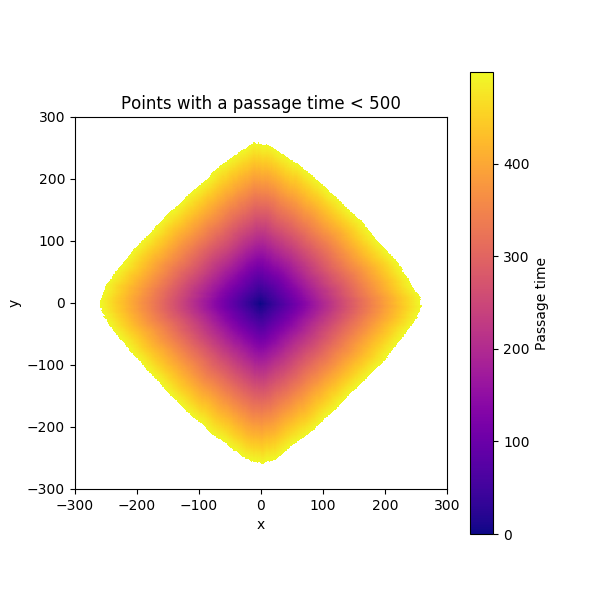}
        
    \end{subfigure}
    \hspace{0.05\textwidth}
    \begin{subfigure}{0.4\textwidth}
        \centering
        \includegraphics[width=\textwidth]{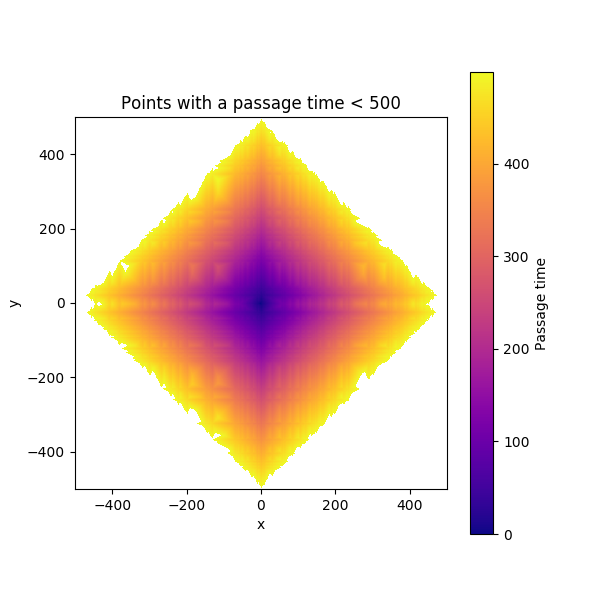}
        
    \end{subfigure}
    \caption{The random balls $B_t$ for  $t=500$ for the distribution $1+Bernoulli(0.95)$ for the classical model (on the left) and the Brochette model (on the right).}
    \label{Figure 1}
\end{figure}
\subsection{The results}
Our first result is that we can identify the time constant in the Brochette model.
 \begin{theo}[Point to point convergence theorem]\label{th:pap}
For Brochette first-passage percolation, if we assume that
\begin{equation}\label{hypo:base}
\mathbb{E}\lbrack \min(\tau_{1},\dots,\tau_{d})\rbrack<+\infty,
\end{equation}
where $\tau_{1},\dots,\tau_{d}$ are  $d$ independent copies with distribution $F$, then we have:
$$
\forall x\in\mathbb{Z}^{d},\text{\;\;\;}\lim_{n\rightarrow+\infty}\frac{T(0,nx)}{n}=a\lVert x\rVert_1 \text{\;\;\;\;\;} a.s. \text{\;and in\;} L^{1},
$$
where we recall that $a$ is the infimum of the support of $\tau$.
\end{theo}
Observe that by the definition of $a$, $T(0,x)\geq a\lVert x\rVert_1$ $a.s.$ so the non trivial part of Theorem \ref{th:pap} is the upper bound. 
\begin{rem}
Let us note that Assumption \eqref{hypo:base} is not the same as in the i.i.d. case, as there are $d$ independent integer lines passing through each vertex instead of $2d$ edges as in the classical model. It follows that \eqref{hypo:base} is optimal. Indeed, if Assumption \eqref{hypo:base} is not satisfied $T(0,e_{1})$ is not integrable.
\end{rem}
Our Theorem \ref{th:pap} states that in the case $a=0$, $T(0,n)$ is sublinear, with the exact order of magnitude
yet to be determined. In the classical model, it is known that $\E[T(0,n)]$ is sublinear as soon as $F(0)\geq p_c$ (see Theorem 1.15 in \cite{Kesten}). Several works tried to better understand the behaviour of $\E[T(0,n)]$ in this case. First Zhang \cite{zhang} established that  $\E[T(0,n)]$ remains bounded if $F(0)>p_c$. The case $F(0)=p_c$ is more subtle. We mention a few results in dimension $d=2$. Chayes, Chayes and Durrett proved in \cite{chayes} that if $F$ is a Bernoulli distribution of parameter $p_c$, then there exist $C_1,C_2>0$ such that $C_1\log(n)\leq\E[T(0,n)]\leq C_2\log(n)$ for all $n$ large enough. Besides Zhang also showed in \cite{zhang2} that there are some distributions $F$ such that $F(0)=p_c$ for which $\E[T(0,n)]$ converges to a non-negative constant (see also \cite{damron} for more related results). 

We will see that, for the Brochette model, depending on the behaviour of the distribution function $F$ of the passage times near 0, a wide variety of behaviours can emerge even in dimension $d=2$.
\begin{theo}\label{th:0tb}
    Consider Brochette first-passage percolation with $a=0$ in dimension $d=2$. If there exist $C>0,\beta\geq0$ such that $F(t)\sim_{t\rightarrow 0+}Ct^{\beta}$, and if the passage times are bounded and atomless, then for all $x\in\Z^{2}$ there exist  $A_1,A_2>0$ such that for all $n\geq2$:

\begin{itemize}
    \item If $\beta<1$, $A_1\leq \E[T(0,nx)]\leq A_2.$
    \item If $\beta=1$, $A_1 \log(n)\leq \E[T(0,nx)]\leq A_2 \log(n).$
    \item If $\beta>1$, $A_1n^{1-\frac{1}{\beta}}\leq \E[T(0,nx)]\leq A_2n^{1-\frac{1}{\beta}}.$
\end{itemize}
\end{theo}
\begin{rem}
We could not obtain comparable results in higher dimensions, which is why we focus on the case $d=2$.
   
\end{rem}
The third result is the shape theorem.
We denote by $\Diamond$ the closed unit ball for the $L^1$ norm.
\begin{theo}\label{th:forme}
   For Brochette first-passage percolation, if we assume that
    
    \begin{equation}\label{condition}
    \E[\min(\tau_{1}^{d},\dots,\tau_{d}^{d})]<+\infty 
    \end{equation}
    
    and that $a>0$ verifies \eqref{hypo:a}, then for all $\epsilon>0$ we have:
    $$
    \P(\frac{1-\epsilon}{a}\Diamond\subset \frac{B_{t}}{t}\subset \frac{1}{a}\Diamond \text{ for all large }t)=1.
    $$
    Furthermore, if $a=0$ we have for all $M>0$:
    $$
     \P(M\Diamond\subset \frac{B_{t}}{t} \text{ for all large } t)=1.
    $$
\end{theo}
\begin{rem}
It will be seen in Proposition \ref{prop:condition} that Condition \eqref{condition} is optimal:  there is no shape theorem if \eqref{condition} does not hold.
\end{rem}
Finally, we consider one last case: when $\P(\tau=+\infty)>0$. In this case, it amounts to studying first-passage percolation on the cluster $\mathcal{C}_{\ast}$ composed of lines whose passage times are finite. Moreover, as soon as $\P(\tau<+\infty)>0$ the obtained subgraph is both connected and infinite whereas it is not the case in the classical model. Indeed, if an edge has a finite passage time, then all the edges on the same line have the same finite passage time. Thus, it is easy to see that the subgraph is a union of integer lines.
In order to state Theorem \ref{th:infstar}, we need to introduce a generalized notion of passage time. 
\begin{defi}
If $x,y\in\Z^{d}$, let $T^{\ast}(x,y)=T(x^{\ast},y^{\ast})$ where $x^{\ast}$ is the closest vertex in $L^{1}$ norm to $x$ in the cluster of edges with finite passage times with an arbitrary rule to break ties. 
\end{defi}
\begin{theo}\label{th:infstar}
For Brochette first-passage percolation, if \;$\P(\tau<+\infty)>0$, we have:

$$
\forall x\in\mathbb{Z}^{d}\;\;\lim_{n\rightarrow+\infty}\frac{T^{\ast}(0,nx)}{n}=a\lVert x \rVert_1\text{\;\;\;in\;probability}.
$$
\end{theo}
This is the analogue of Theorem 4 in \cite{Theret}, we follow the same structure of proof. 
An important difference is that \cite{Theret} relies on non-trivial estimates in classical percolation. In our case, the infinite cluster has a rather simple shape, which allows for self-contained proofs.
To conclude, we outline how the different sections of this article will be structured.
\begin{itemize}
    \item In Section \ref{Section2} we focus on point to point convergence Theorem. We prove Theorem \ref{th:pap}.
    \item In Section \ref{Section 3} we prove Theorem \ref{th:0tb} about the variety of behaviours for the expectation of the passage time in dimension $d=2$.
    \item In Section \ref{Section 4} we obtain a shape theorem by proving Theorem \ref{th:forme}.
    \item In Section \ref{Section 5} we prove Theorem \ref{th:infstar} about the convergence in probability of $\frac{T^{\ast}(0,nx)}{n}$.
\end{itemize}

\section{Point to point almost sure convergence}\label{Section2}
This section is devoted to the proof of Theorem \ref{th:pap} : the $a.s.$ convergence of $\frac{T(0,nx)}{n}$. In addition to prove the linear growth of the passage time we are able to determine the limit constant which is equal to $a \lVert x \rVert_1$. This is a first step towards understanding the behaviour of $T(0,x)$ in the context of the Brochette first-passage percolation. This result will be improved in Section \ref{Section 4}.
\begin{proof}[Proof of Theorem \ref{th:pap}] 
 To prove this theorem we will use the following version of Kingman's theorem:
\begin{theo}[\textbf{Subadditive Ergodic Theorem \cite{Kingman}}]\label{th:Kingman} Let $(X_{m,n})_{0\leq m\leq n}$ be a family of random variables that satisfy the following assumptions:
\begin{enumerate}[label=\alph{*})]
\item $X_{0,n}\leq X_{0,m}+X_{m,n}$ for all $0<m<n$.
\item The sequences $(X_{m,m+k})_{k\geq1}$ and $(X_{m+1,m+k+1})_{k\geq1}$ have the same distribution for all $m\geq0$.
\item For all $k\geq1$, the sequence $(X_{nk,(n+1)k})_{n\geq0}$ is stationary.
\item $\mathbb{E}\lbrack X_{0,1} \rbrack<+\infty$ and $\mathbb{E}\lbrack X_{0,n} \rbrack>-cn$ for some finite constant $c$. 
\end{enumerate}
Then there exists  a random variable $X$ such that:
$$
\frac{X_{0,n}}{n}\rightarrow_{n\rightarrow+\infty}X \text{\;\;\;} a.s. \text{\;and in\;} L^{1}. 
$$
\end{theo}
\begin{figure}
\centerline{\includegraphics[scale=0.23]{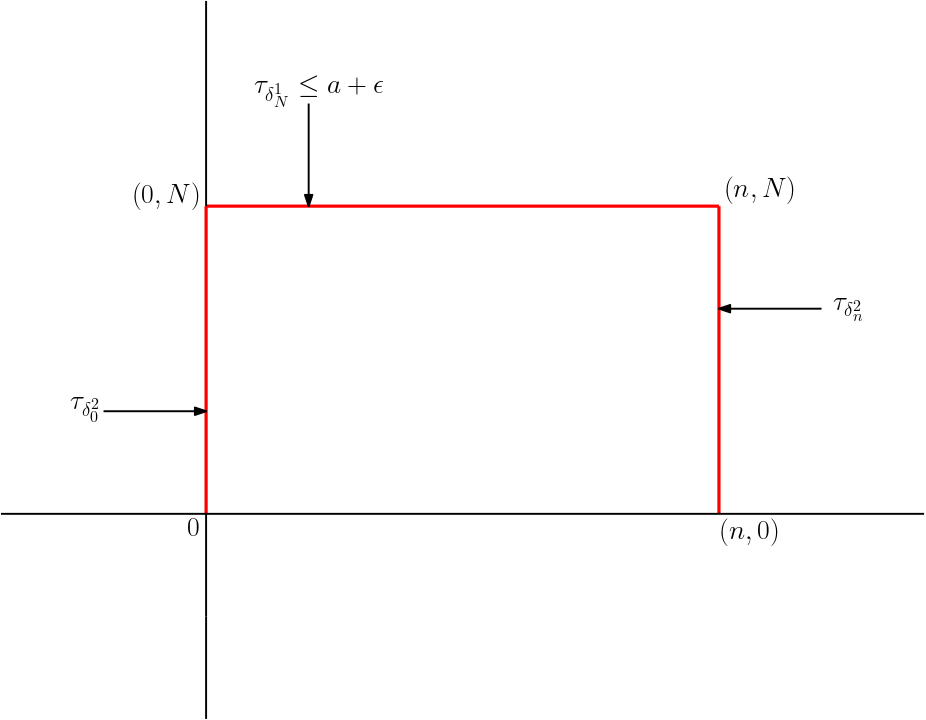}}~
\caption{The path going from $0$ to $ne_1$ following the construction.}
\label{Figure 2}
\end{figure}
Fix $x\in\mathbb{Z}^{d}$. We aim to apply Theorem \ref{th:Kingman} to $X_{m,n}=T(mx,nx)$. We verify briefly that $a),b),c),d)$ hold. $a)$ is immediate. As in the i.i.d case the environment is invariant under shifts of $\Z^{d}$ so assumptions $b)$ and $c)$ hold. Moreover, $\E[X_{0,n}]>-cn$ because passage times are non-negative. 
It remains to prove $E[X_{0,1}]<+\infty$.
There exist $d$ paths $\Gamma_1,\dots,\Gamma_d$ from $0$ to $x$ whose passage times are independent. Thus, $T(0,x)\leq \min(T(\Gamma_1),\dots,T(\Gamma_d))$ and if we let $L$ the length of the longest of these paths (say $\Gamma_1$) we have:
 \begin{align*}
     \P(T(0,x)>s)&\leq \P(\bigcap_{i=1}^{d}\{T(\Gamma_i)>s\})\\
                   &=\prod_{i=1}^{d}\P(T(\Gamma_i)>s)\\
                   &\leq \P(T(\Gamma_1)>s)^d\\
                   &=(L\P(\tau_e>\frac{s}{L}))^d.
 \end{align*}
 Therefore, considering $Y:=\min(\tau_1,\dots,\tau_d)$ where $\tau_1,\dots,\tau_d$ are i.i.d. with the same distribution $F$, we have:
 $$
\P(T(0,x)>s)\leq L^d\P(Y>\frac{s}{L}),
 $$
 from which we deduce, thanks to \eqref{hypo:base}, that $\E[X_{0,1}]<+\infty$.
Kingman's theorem therefore implies that there exists a random variable $X_x$ such that:
$$
\lim_{n\rightarrow+\infty}\frac{T(0,nx)}{n}=X_x \; a.s. \text{ and in } L^{1}.
$$
We first prove that for $x=e_1$, it suffices to show that:
$$
\forall \epsilon >0, \liminf \frac{T(0,ne_1)}{n}\leq a+\epsilon\; a.s..
$$
Let $\epsilon>0$. We denote by $\delta_k^1$ the line $ke_2+\Z e_1$ with $e_2:=(0,1,0,\dots,0)$. Let $N=\inf\{ k\geq 1, \tau_{\delta_k^1}\leq a+\epsilon\}$. Since $(\delta_k^1)_{k\in \N}$ are i.i.d. and by definition of $a$, $N<+\infty \; a.s.$. Let $\delta_k^2$ be the line $ke_1+\Z e_2$ (\text{See Figure \ref{Figure 2}}). Then:
$$
T(0,ne_1)\leq N \tau_{\delta_0^2}+ n(a+\epsilon)+ N\tau_{\delta_n^2}.
$$ 
Thus
$$
\frac{T(0,ne_1)}{n}\leq \frac{N}{n}\tau_{\delta_0^2}+(a+\epsilon)+ \frac{N}{n}\tau_{\delta_n^2}.
$$
Since $N<+\infty$ almost surely, doesn't depend on $n$ and $\tau_{\delta_0^2}<+\infty$ almost surely too, we have
$\underset{n\to+\infty}{\lim} \frac{N}{n}\tau_{\delta_0^2}=0$. Moreover, 
since $(\tau_{\delta_n^2})_n$ are i.i.d. with distribution $F$, they are infinitely often smaller than $a+1$ and we have 
$
\underset{n\to+\infty}{\liminf} \frac{N}{n}\tau_{\delta_n^2}=0.
$
We conclude that
$$
\liminf_{n\to+\infty} \frac{T(0,ne_1)}{n}\leq a+\epsilon\; a.s.. 
$$
Therefore we have
$$
\lim_{n\to+\infty} \frac{T(0,ne_1)}{n}=a \; a.s.\text{ and in\;}L^1.
$$
Now, let $x=(x_1,\dots,x_d)\in\N^d$ (the proof for $x\in\Z^d$ is the same). Let $y_i=(x_1,\dots,x_i,0,\dots,0)$. Then:
$$
T(0,nx)\leq \sum_{i=0}^{d-1} T(ny_i,ny_{i+1}).
$$
Since $T(ny_i,ny_{i+1})\overset{(d)}{=}T(0,nx_{i+1}e_{i+1})$ we have $\frac{T(ny_i,ny_{i+1})}{n}\underset{n\rightarrow+\infty}{\rightarrow}ax_{i+1}$ in probability.
Thus, we get:
$$
\lim_{n\to+\infty} \frac{T(0,nx)}{n}\leq a\sum_{i=0}^{d-1} x_{i+1}=a \lVert x \rVert_1.
$$
\end{proof}

\section{The case $a=0$}\label{Section 3}
In this section, we study the case where $a=0$ and $d=2$. We will see that, depending on $F$, a diverse range of behaviours for $T(0,n)$ can arise.
For this purpose we define the following random variable.
\begin{defi}
    If $\tau_{1},...,\tau_{n}$ are independent passage times we denote:
    $$
M_{n}:=\min(\tau_{1},...,\tau_{n}).
    $$
\end{defi}
  To prove Theorem \ref{th:0tb}, we will see it as a corollary of the following theorem.

\begin{theo}\label{th:0} For the Brochette first-passage percolation when $d=2$, if $a=0$ with $a$ defined by \eqref{hypo:a}, if the distribution of $\tau$ is atomless and if we assume that $\E[min(\tau_{1},\tau_{2})]<+\infty$ then we have:
$$
\forall x\in \Z^{2},\;\;\; \sum_{k=0}^{\lVert nx \rVert_1-1}\E[M_{4k+2}]\leq \E[T(0,nx)]\leq 8\sum_{k=0}^{\lVert nx\rVert_1-1}\E[M_{k}].
$$
\end{theo}
 \begin{proof}[Proof of Theorem \ref{th:0}]
  Let us denote $X_{n}$ the random variable which is equal to the abscissa of the vertical edges with the smallest passage time in $\llbracket 0;n\rrbracket^{2}$ and $Y_{n}$ the random variable which is equal to the ordinate of the horizontal edges with the smallest passage time in $\llbracket 0;n \rrbracket^{2}$. We denote $I_{n}:=T(0,(X_{n},Y_{n}))$, with $I_0:=0$. Then by sub-additivity, we have:
$$
I_{n+1}-I_{n}\leq T((X_{n},Y_{n}),(X_{n+1},Y_{n+1})),
$$
which gives:
$$
\mathbb{E}[I_{n+1}-I_{n}]\leq \mathbb{E}[T((X_{n},Y_{n}),(X_{n+1},Y_{n+1}))].
$$

Now if $(X_{n+1},Y_{n+1})\neq (X_{n},Y_{n})$, we have:
$$
(X_{n+1},Y_{n+1})\in\{(X_{n},n+1);(n+1,Y_{n});(n+1,n+1)\}.
$$
Moreover, $X_{n}$ and $Y_{n}$ are independent and since the passage times are atomless we have:
$$
\P(X_{n+1}=n+1)=\P(Y_{n+1}=n+1)=\frac{1}{n+2}.
$$
Let $T_{n,n+1}:=T((X_{n},Y_{n}),(X_{n+1},Y_{n+1}))$. If $n\in \N$ we denote by $\tau_{v,n}$ the passage time associated with the vertical integer line of abscissa $n$ and  by $\tau_{h,n}$ the passage time associated with the horizontal integer line of ordinate $n$.  We have:
\begin{enumerate}
    \item If $(X_{n+1},Y_{n+1})=(X_{n},n+1)$, $T_{n,n+1}\leq (n+1)\tau_{v,X_{n}}$.
    \item If $(X_{n+1},Y_{n+1})=(n+1,Y_{n})$, $T_{n,n+1}\leq (n+1)\tau_{h,Y_{n}}$.
    \item If $(X_{n+1},Y_{n+1})=(n+1,n+1)$, $T_{n,n+1}\leq (n+1)\tau_{v,X_{n}}+(n+1)\tau_{h,n+1}$.
\end{enumerate}
Finally we have, since $X_n$ is independent of the horizontal passage times and $Y_n$ is independent of the vertical passage times:\\
 \begin{align*}
     \E[I_{n+1}-I_{n}]&\leq\E[T_{n,n+1}]\\
                      &=\E[T_{n,n+1}\mathds{1}_{(X_{n+1},Y_{n+1})=(X_{n},n+1)}]+ \E[T_{n,n+1}\mathds{1}_{(X_{n+1},Y_{n+1})=(n+1,Y_{n})}]\\&\;\;\;\;+\E[T_{n,n+1}\mathds{1}_{(X_{n+1},Y_{n+1})=(n+1,n+1)}]\\
                      &\leq \E[(n+1)\tau_{v,X_{n}}\mathds{1}_{(X_{n+1},Y_{n+1})=(X_{n},n+1)}]+\E[(n+1)\tau_{h,Y_{n}}\mathds{1}_{(X_{n+1},Y_{n+1})=(n+1,Y_{n})}]\\
                      &\;\;\;\;+\E[((n+1)\tau_{v,X_{n}}+(n+1)\tau_{h,n+1})\mathds{1}_{(X_{n+1},Y_{n+1})=(n+1,n+1)}]\\
                      &\leq\E[\tau_{v,X_{n}}\mathds{1}_{X_{n+1}=X_{n}}]+\E[\tau_{h,Y_{n}}\mathds{1}_{Y_{n+1}=Y_{n}}]+\E[\tau_{v,X_{n}}\mathds{1}_{X_{n+1}=n+1}]+\E[\tau_{h,n+1}\mathds{1}_{Y_{n+1}=n+1}]\\
                      &\leq2\E[M_{n}].
 \end{align*}
Hence, by summing, we obtain:
$$
\mathbb{E}[I_{n}]\leq 2 \sum_{k=0}^{n-1} \mathbb{E}[M_{k}].
$$
Moreover, to go from $0$ to $(n,0)$, we can construct with the same method a path from $(n,0)$ to $(X_{n},Y_{n})$. We can do the same computation and obtain that:
$$
    \mathbb{E}[T(0,(n,0))]\leq 2\mathbb{E}[I_{n}].
$$
Now, to go from $0$ to $x=(x_{1},x_{2})$, we can use the same method twice to go from $0$ to $(x_{1},0)$ and then from $(x_{1},0)$ to $(x_{1},x_{2})$. It follows that by telescoping we have:
\begin{align*}
    \E[T(0,nx)]&\leq 4\E[I_{\lVert nx\rVert_1 }]\leq8\sum_{k=0}^{\lVert nx\rVert_1-1}\E[M_{k}].
\end{align*}
For the lower bound we proceed differently. We note that to go from $0$ to $nx$, a path must cross all the deterministic $L^1$ balls of radius $k$ for $k$ between $0$ and $\lVert nx\rVert_1-1$. Thus, the passage time from $0$ to $nx$ is greater than the sum of the minimum passage times among the edges of the sphere of radius $k$ for $k$ ranging from $0$ to $\lVert nx\rVert_1-1$. Moreover, in dimension $2$, $4k+2$ independent edges originate from a vertex on the sphere of radius $k$. Thus, we have:
$$
\sum_{k=0}^{\lVert nx \rVert_1-1}\E[M_{4k+2}]\leq \E[T(0,nx)].
$$
\end{proof}
\begin{proof}[Proof of Theorem \ref{th:0tb}]
 To prove Theorem \ref{th:0tb}, we apply the estimate of Lemma \ref{lem:tpsborne} to Theorem \ref{th:0}.
\end{proof}
\section{Shape Theorem}\label{Section 4}

This section is devoted to the proof of Theorem \ref{th:forme}. First, by the definition of our model, $\frac{B_{t}}{t}\subset \frac{1}{a}\Diamond$ for all $t$. Only the first inclusion remains to be proven. To improve a pointwise convergence into a shape theorem, all usual proofs rely at some point on a uniform control of $T(x,y)$. To obtain such control, we prove here the analogue of the Difference-Estimate Lemma (Lemma 2.20, \cite{Survey}). This step presents a particular difficulty in the "Brochette" case.
This step follows the same strategy as the one developed in \cite{Cox} but presents a particular difficulty in the "Brochette" case. As
we will see, the dependence makes the construction of certain paths more delicate.
\begin{lem}\label{lem:unif}
Consider Brochette first-passage percolation and assume that we have the integrability condition \eqref{condition}. Then, there exists a constant $\kappa>0$ such that, for every $z\in\Z^{d}$,
$$
\P(\sup_{\substack{x\in\Z^{d}\\ x\neq z}}\frac{T(z,x)}{\lVert x-z \rVert_1}<\kappa)>0.
$$
\end{lem}
\begin{proof}
By shift invariance, it is sufficient to prove the result for $z=0$.
We provide all the details of the proof in the case $d=2$ and we will briefly explain at the end how to generalize it to an arbitrary $d$.
Set $x=(x_1,x_2)$. Without loss of generality, we assume that $x_1\geq x_2\geq 0$. Similarly to classical proofs, the estimates will rely on the construction of paths from $0$ to $x$. 
\begin{itemize}
    \item We first construct $(d+1)=3$ paths $\Gamma_1, \Gamma_2, \Gamma_3$ that primarily use edges whose passage times are independent of most of the edges in the other paths.
    \item Then for $1\leq i\leq 3$, we will bound $\mathrm{Var}(T(\Gamma_i))$. To do this, we will decompose each $\Gamma_i$ into several segments, which we will call steps, and most of these are independent.
    \item In the last step of the proof, we exploit the fact that $\Gamma_1, \Gamma_2$ and $\Gamma_3$ are mostly independent in order to control $T(0,x)$.
\end{itemize}

\paragraph {Construction of paths $\Gamma_{1},\Gamma_2,\Gamma_3$.}(see Figure \ref{Figure 3}). We begin by defining a deterministic path  $S_{1}$ from $0$ to $x$, which we will refer to as the skeleton. It is composed of three staircases: it descends and then ascends to $(x_1-x_2,0)$ before ending at $(x_1,x_2)$. Formally, if $p=\lfloor \frac{x_1-x_2}{15}\rfloor$ and $p'=\lfloor\frac{x_2}{15}\rfloor$, $S_1$ is the path that passes through the following points:
\begin{align*}
&x^{1}_0=(0,0)\rightarrow x^{1}_1=(0,-15)\rightarrow x^{1}_2=(15,-15)\rightarrow\dots\rightarrow x^{1}_{2\lceil \frac{p}{2}\rceil}=(\lceil\frac{p}{2}\rceil 15,-\lceil\frac{p}{2}\rceil 15)\\&
\rightarrow x^{1}_{2\lceil \frac{p}{2}\rceil+1}=(\lceil\frac{p}{2}\rceil 15,-(\lceil\frac{p}{2}\rceil-1)15)\rightarrow x^{1}_{2\lceil \frac{p}{2}\rceil+2}=((\lceil\frac{p}{2}\rceil+1) 15,-(\lceil\frac{p}{2}\rceil-1)15)\\&
\rightarrow\dots\rightarrow x^{1}_{m_1}=(x_1-x_2,0)\rightarrow x^{1}_{m_1+1}=(x_1-x_2,15)\\&\rightarrow x^{1}_{m_1+2}=(x_1-x_2+15,15)\rightarrow\dots\rightarrow x^{1}_{m_1+n_1}=(x_1,x_2),
\end{align*}
and such that between two consecutive points $x^{1}_i$, it follows the straight line. Note that $x^{1}_i$ and $x^{1}_{i+1}$ are almost always at most 15 edges apart. Depending on the value of $p$ and $p'$, $m_1\in \{4\lceil \frac{p}{2}\rceil-1;4\lceil \frac{p}{2}\rceil\}$ and $n_1\in\{2p';2p'+2\}$. Similarly, we construct two other skeletons with the same shape but slightly offset from $S_1$ for independence reasons that we will see later in the proof.
\begin{itemize}
\item $S_2$ goes through:\\
     $x^{2}_0=(0,0)\rightarrow x^{2}_1=(0,-10)\rightarrow x^{2}_2=(10,-10)\rightarrow x^{2}_3=(10,-25)\rightarrow\dots\rightarrow x^{2}_{2\lceil \frac{p}{2}\rceil}=(10+(\lceil\frac{p}{2}\rceil-1) 15,-10-(\lceil\frac{p}{2}\rceil-1) 15)\rightarrow x^{2}_{2\lceil \frac{p}{2}\rceil+1}=(10+(\lceil\frac{p}{2}\rceil-1) 15,-10-(\lceil\frac{p}{2}\rceil-2) 15)\rightarrow\dots\rightarrow x^{2}_{4\lceil \frac{p}{2}\rceil-1}=((2\lceil\frac{p}{2}\rceil-1)15-5,0)\rightarrow x^{2}_{4\lceil \frac{p}{2}\rceil=:m_2}=(x_1-x_2,0)
     \rightarrow x^{2}_{m_2+1}=(x_1-x_2,10)\\\rightarrow x^{2}_{m_2+2}=(x_1-x_2+10,10)\rightarrow x^{2}_{m_2+3}=(x_1-x_2+10,10+15)\rightarrow x^{2}_{m_2+4}=(x_1-x_2+10+15,10+15)\rightarrow\dots\rightarrow x^{2}_{m_2+n_2}=(x_1,x_2)$,\\
     with $n_2\in\{2p'+2;2p'+4\}.$
      \item $S_3$ goes through:\\
     $x^{3}_0=(0,0)\rightarrow x^{3}_1=(0,-5)\rightarrow x^{3}_2=(5,-5)\rightarrow x^{3}_3=(5,-20)\rightarrow\dots\rightarrow x^{3}_{2\lceil\frac{p}{2}\rceil}=(5+(\lceil\frac{p}{2}\rceil-1) 15,-5-(\lceil\frac{p}{2}\rceil-1) 15)\rightarrow x^{3}_{2\lceil\frac{p}{2}\rceil+1}=(5+(\lceil\frac{p}{2}\rceil-1) 15,-5-(\lceil\frac{p}{2}\rceil-2) 15)\rightarrow\dots\rightarrow x^{3}_{4\lceil\frac{p}{2}\rceil-1}=((2\lceil\frac{p}{2}\rceil-1)15-10,0)\rightarrow x^{3}_{4\lceil\frac{p}{2}\rceil=:m_3}=(x_1-x_2,0)
     \rightarrow x^{3}_{m_3+1}=(x_1-x_2,5)\\\rightarrow x^{3}_{m_3+2}=(x_1-x_2+5,5)\rightarrow x^{3}_{m_3+3}=(x_1-x_2+5,5+15)\rightarrow x^{3}_{m_3+4}=(x_1-x_2+5+15,5+15)\rightarrow\dots\rightarrow x^{3}_{m_3+n_3}=(x_1,x_2)$,\\
     with $n_{3}\in\{2p'+2;2p'+4\}.$
     \end{itemize}
A maximal sequence of consecutive edges in the same direction within $S_\ell$ is called a \emph{step}. (We see from the construction that steps consist of at most 39 edges but most often contain only 15.) 
\begin{figure}
\centerline{\includegraphics[scale=0.3]{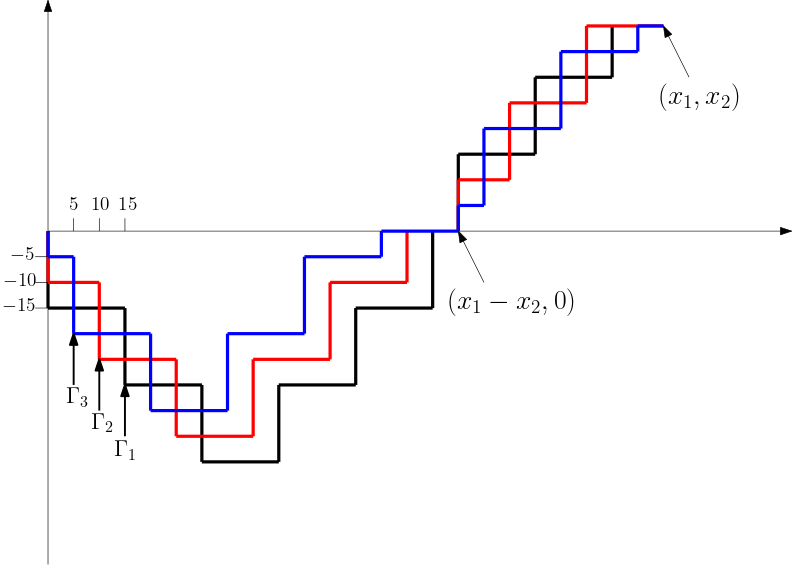}}~
\caption{Paths $\Gamma_1,\Gamma_2,\Gamma_3$.}
\label{Figure 3}
\end{figure}\\
\begin{defi}
 For $1\leq \ell\leq 3$ let $s$ be a step of $S_\ell$.  Assume $s$ is a horizontal (resp. vertical) step. Let $a=(a_1,a_2)$ and $b=(b_1,a_2)$ (resp. $a=(a_1,a_2)$ and $b=(a_1,b_2))$ be the endpoints of $s$.  Then the \emph{envelope} of $s$ is  defined by the set of edges forming the borders of the rectangle with vertices
$$(a_1,a_2),(b_1,a_2),(b_1,a_2+1),(a_1,a_2+1)$$
(resp.$$(a_1,a_2),(a_1+1,a_2), (a_1+1,b_2),(a_1,b_2)).$$
\end{defi}
In the above definition, if $a$ and $b$ are the vertices $x^{\ell}_i $ and $x^{\ell}_{i+1}$ in the traversal of $S_\ell$, the envelope of the step $a,b$ is referred to as the $i$-th envelope.
\begin{lem}\label{lem:envelope}
    Let $S_\ell$ be a skeleton and $E$ the envelope of a step of $S_\ell$. Then $\{\tau_{e},e\in E\}$ is independent of $\{\tau_{e},e\in E'\}$ for all envelope $E'$ of $S_\ell$ but at most $10$ other envelopes $E'$.
\end{lem}
\begin{proof}[Proof of lemma \ref{lem:envelope}]
    If the passage times of the edges in an envelope depend on those of $E$, then this envelope must necessarily touch one of the two strips defined by the sides of $E$ (see Figure \ref{Figure 4}). Furthermore, based on the construction of our paths and by looking at Figure \ref{Figure 4}, we can see that there are at most ten such envelopes.
\begin{figure}
\centerline{\includegraphics[scale=0.25]{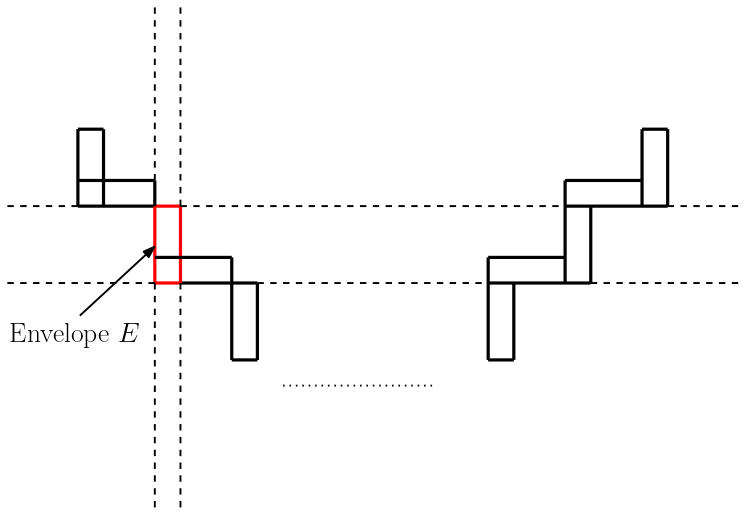}}~
\caption{Envelopes which are dependent of the red one.}
\label{Figure 4}
\end{figure}\\
\end{proof}
We are now ready to build $\Gamma_\ell$. It is a random path connecting $0$ to $x$ and contained in the set of envelopes of the steps of $S_\ell$. Let $\gamma^{\ell}_i$ be the path with the shortest passage time among the two paths connecting $x_i^{\ell}$ and $x_{i+1}^{\ell}$ in the $i$-th envelope of $S_\ell$.
Then we define $\Gamma_\ell$ as the concatenation\footnote{It is possible that the path $\Gamma_\ell$ passes through the same vertex several times.} of $\gamma^{\ell}_1,\gamma^{\ell}_2,\dots,\gamma^{\ell}_{k_{\ell}}$ where $k_{\ell}=m_{\ell}+n_{\ell}$ is the number of steps of $S_\ell$.
Notice that the distribution of $T(\gamma^{\ell}_{i})$ only depends on the number of edges of the $i$-th step of $S_\ell$. We can then choose $\tilde{T}$ such that $\tilde{T}\overset{(d)}{=}\min(39\tau_1,\tau_2+39\tau_3+\tau_4)$.
 Thus, for all $i$ and $\ell$, $\tilde{T}$ stochastically dominates $T(\gamma^{\ell}_{i})$. 
\begin{lem}\label{lem:concentration}
    For $1\leq \ell \leq 3$ and $1\leq i\leq k_\ell$ set $T_i^{\ell}=T(\gamma_i^{\ell})$.
    Then for all $1\leq k\leq k_\ell$ there exists $C>0$ such that we have:
    $$
\P(\sum_{i=1}^{k}T_{i}^{\ell}>k(\E[\tilde{T}]+1))\leq \frac{C}{k}.
    $$
\end{lem}
\begin{proof}
    Using straightforward calculation (see  eq (2.3) in \cite{Survey}) and thanks to Assumption \eqref{condition}, for all $i$ and $l$, $\mathrm{Var}(T^\ell_i)<C'$ for some $C'$ which does not depend on $k,i$ and $\ell$. Using Lemma \ref{lem:envelope}, we have:
    $$
   \mathrm{Var}(\sum_{i=1}^{k}T_i^\ell) =\sum_{1\leq i,j\leq k}\mathrm{Cov}(T_i^\ell,T_j^\ell)\leq \sum_{i=1}^{k} 10 \max_j \mathrm{Cov}(T_i^\ell,T_j^\ell).
$$
Moreover, if $\mathrm{Cov}(T_i^\ell,T_j^\ell)\neq0$, then we have, according to the Cauchy-Schwarz inequality:
$$
\mid \mathrm{Cov}(T_i^\ell,T_j^\ell)\mid
\leq \sqrt{\text{Var}(T_i^\ell)\text{Var}(T_j^\ell)}\
\leq C'.
$$
Consequently, we have:
$$
\mathrm{Var}(\sum_{i=1}^{k}T_i^\ell)\leq 10kC'.
$$
Thus, since by definition $\E[\tilde{T}]\geq\E[T_i^\ell]$ for all $1\leq i\leq k$ we have thanks to Tchebytchev inequality:
\begin{align*}
\P(\sum_{i=1}^{k} T_{i}>k(\E[\tilde{T}]+1))\leq \P(\sum_{i=1}^{k} T_i^\ell>\sum_{i=1}^{k} \E[T_i^\ell]+k)&\leq\frac{10}{k}C'.
\end{align*}
Denoting $C=10C'$, the result follows.
\end{proof}
In order to conclude the proof, we need to control $T(0,x)$ using the independence of $\Gamma_1, \Gamma_2$ and $\Gamma_3$.
By construction, $\Gamma_1, \Gamma_2$ and $\Gamma_3$ involve independent envelopes, except in the neighborhoods of $0$, $(x_1-x_2,0)$ and $x$. In order to exploit this rigorously we introduce a few notation:
\begin{itemize}
\item $U_{\ell}=\sum_{i=3}^{m_\ell-2}T_{i}^{\ell}+\sum_{i=m_\ell+3}^{k_\ell-2}T_{i}^{\ell}$
\item $R=\max_{1\leq \ell\leq 3}(T_{1}^{\ell}+T_{2}^{\ell}+\sum_{i=m_\ell-1}^{m_\ell+2}T^{\ell}_{i}+T_{k-1}^{\ell}+T_{k}^{\ell}).$
\end{itemize}
We note that in this way, the $U_\ell$ are independent. Indeed, by construction, the envelopes of $S_\ell$ on which $U_\ell$ depends consist of edges located on integer lines distinct from those forming the envelopes of $S_\ell'$ on which $U_\ell'$ depends, for all distinct $\ell,\ell'\in\{1;2;3\}$.
Moreover, we have thanks to the union bound:

$$
\P(R>\lVert x \rVert_1)\leq 24 \P(\tilde{T}>\frac{\lVert x \rVert_1}{8}).
$$
Additionally, we notice that $T(0,x)\leq R+\min_{1\leq \ell \leq 3}U_{\ell}$. Therefore, we have with the union bound:

\begin{align*}
\P(T(0,x)>(\E[\tilde{T}]+2)\lVert x \rVert_1)&\leq \P(R>\lVert x \rVert_1)+ \P(\min_{1\leq \ell\leq 3}U_{\ell}>(\E[\tilde{T}]+1)\lVert x \rVert_1)\\& \leq 24 \P(\tilde{T}>\frac{\lVert x \rVert_1}{8})+\P(\min_{1\leq \ell\leq 3}U_{\ell}>(\E[\tilde{T}]+1)\lVert x \rVert_1).
\end{align*}

Now according to Lemma \ref{lem:concentration} $\P(U_{\ell}>k_\ell(\E[\tilde{T}]+1))\leq \frac{C}{k_\ell}$.
Thus, by the independence of the $U_{\ell}$, we have, since each path is composed of less than $\lVert x \rVert_1$ steps but more than $\frac{\lVert x \rVert_1}{15}$:
$$
\P(T(0,x)>(\E[\tilde{T}]+2)\lVert x \rVert_1)\leq 24 \P(\tilde{T}>\frac{\lVert x \rVert_1}{8})+ (15C)^{3}\lVert x \rVert_1^{-3}.
$$
Thus, if $K=\E[\tilde{T}]+2$, 
$$
\sum_{x\in\Z^2}\P(\frac{T(0,x)}{\lVert x \rVert_1}>K)\leq 24 \sum_{x\in\Z^2}\P(\tilde{T}>\frac{\lVert x \rVert_1}{8})+ (15C)^{3}\sum_{x\in\Z^2}\lVert x \rVert_1^{-3}.
$$
Thanks to Assumption \eqref{condition}, $\E[\tilde{T}^d]=\E[\tilde{T}^2]<+\infty$ and so $\sum_{x\in\Z^2}\P(\tilde{T}>\frac{\lVert x \rVert_1}{8})<+\infty$.

Therefore, since the series of $\lVert x \rVert_1^{-3}$ over $\Z^{2}$ converges, according to the Borel-Cantelli Lemma there exists $\kappa>0$ such that for all $z\in\Z^{2}$
$$
\P(\forall x\in\Z^{2}\setminus \{z\}, \frac{T(z,x)}{\lVert z-x \rVert_1}<\kappa)>0.
$$
\end{proof}
Everything we have done here can be generalized in dimension $d\geq3$. Indeed the biggest difference is that we need to construct $d+1$ paths. This construction is achievable by simply spacing out the paths and moving along orthogonal planes to maintain as much independence as possible. Moreover, we then consider $d$-dimensional envelopes and denote $\gamma_i^\ell$ the path with the shortest passage time among the $d$ paths connecting $x_i^\ell$ and $x_{i+1}^\ell$ in the $i$-th  $d$-dimensional envelope of $S_\ell$. Thanks to Assumption \eqref{condition}, we have in particular $\E[\min(\tau_1,\dots,\tau_d)^2]<+\infty$, so $\mathrm{Var}(T_i^\ell)$ are still bounded and with obvious notations, we get $\P(\min_{1\leq\ell\leq d+1}U_\ell>(\E[\tilde{T}]+1)\lVert x \rVert_1)\leq C'\lVert x \rVert_1^{-(d+1)}$. Finally, using that $\E[\min(\tau_1,\dots,\tau_d)^d]<+\infty$, we also get that $\sum_{x\in\Z^d}\P(\tilde{T}>\frac{\lVert x \rVert_1}{8})<+\infty$.
\begin{proof}[Proof of Theorem \ref{th:forme}]
We can now use Lemma \ref{lem:unif} to prove Theorem \ref{th:forme} by following the proof in \cite{Survey}.
The translation $f$ with vector $\zeta\in\Z^d$ is not ergodic when $\zeta$ in on one of the axes.  
However, 
ergodicity outside the axes is sufficient to prove Theorem \ref{th:forme}.
Indeed, by carefully reading the proof of  Theorem $2.16$ in \cite{Survey}, just after formula $(2.11)$, we see that ergodicity along directions arbitrarily close to the axes is sufficient. 
The result follows. 
\end{proof}
\begin{prop}\label{prop:condition}
    Assumption \eqref{condition} is necessary for $B_t$ to have a limiting shape. More precisely if \;$\E[\min(\tau_{1}^{d},...,\tau_{d}^{d})]=+\infty$ then for any $K>0$,
    $$
     \frac{T(0,x)}{\lVert x \rVert_1}>K \text{ for infinitely many }x\in \Z^d, a.s.
    $$
\end{prop}
\begin{proof}
Recall that $a$ is the infimum of the support of the distribution of the passage times. Let $K>a$. Let $Q=\{z=(z_1,\dots,z_d)\in\N^d \text{\;such that \;}z_2,\dots,z_d\leq z_1\}$. Let for all $z\in Q$, $A_{z}:=\{\forall i\in \llbracket 1;d \rrbracket, \tau_{z+e_i}\geq dz_1K\}$. If $A_z$ occurs, $T(0,z)\geq K\lVert z \rVert_1$. Then if for any $K>a$ an infinite number of $A_z$ occur, there is no shape theorem.
To prove that  this is the case with probability 1 if $\E[\min(\tau_{1}^{d},\dots,\tau_{d}^{d})]<+\infty$, we use the following Borel-Cantelli lemma:
\begin{lem}\label{lem:BC}\textbf{\cite{Lem:BC} Lemma C).}
Let $(A_n)_{n\geq1}$ be such that $\sum_{n=1}^{+\infty}\P(A_n)=+\infty$ and such that:
            $$
\limsup_{n\rightarrow+\infty}\frac{(\sum_{i=1}^{n}\P(A_i))^{2}}{\sum_{1\leq i,j\leq n}\P(A_i\cap A_j)}=1.
            $$
            Then $\P(\limsup A_n)=1$.
        \end{lem}

In our case, we want to prove that
$\limsup_{n\rightarrow+\infty}\frac{(\sum_{z\in Q\cap\llbracket 1:n \rrbracket^d}\P(A_z))^2}{\sum_{z,z'\in Q\cap\llbracket 1:n \rrbracket^d}\P(A_z\cap A_z')}=1.$
We have $\P(A_z)=p_{z_1}^{d}$, where we put $p_i:=\P(\tau\geq dKi)$. 
    Let 
    $$
    V_n:=(\sum_{z\in Q\cap\llbracket 1;n \rrbracket^{d}}\P(A_z))^{2}=(\sum_{i=1}^{n}i^{d-1}p_i^{d})^{2}.
    $$
    Now by assumption $\E[\min(\tau_1,\dots,\tau_d)^{d}]=+\infty$ and by series-integral comparison there exists a constant $R>0$  such that
    $\sum_{i=1}^{+\infty}i^{d-1}p_i^{d}\geq R\E[\min(\tau_1,\dots,\tau_d)^{d}]$.
    So 
    $$
    \sum_{i=1}^{+\infty}i^{d-1}p_i^{d}=+\infty.
    $$
    Let $z=(z_1,\dots,z_d), z'=(z'_1,\dots,z'_d)\in Q$ and let us denote $I(z,z'):=\{j\in\llbracket 1;d \rrbracket:z_j=z'_j\}$. We have:

$\P(A_z\cap A_{z'})=
\begin{cases}
    
    &\P(A_z)\P(A_{z'}) \;\;\; if\mid I \mid=0\\ 
   & p_{z_1}^{2d-k} \;\;\;if \mid I \mid=k, k>0,1\in I\\
    & p_{\min(z_1,z_1')}^{d-k}p_{\max(z_1,z_1')}^{d}\;\;\; if
    \mid I \mid=k, k>0,1\notin I.
\end{cases}
$
\\
    Then:
    \begin{align*}
    \sum_{\substack{z,z'\in Q\cap\llbracket 1;n\rrbracket^{d}}}
    \P(A_z\cap A_{z'})&\leq\sum_{\substack{z,z'\in Q\cap \llbracket 1;n\rrbracket^{d}\\ \mid I \mid=0}}\P(A_z)\P(A_{z'})+\sum_{k=1}^{d}\sum_{\substack{z,z'\\\mid I(z,z')\mid=k\\ 1\in I(z,z')}}p_{z_1}^{2d-k}+\sum_{k=1}^{d-1}\sum_{\substack{z,z'\\\mid I(z,z') \mid=k\\1\notin I(z,z')}}p_{\max(z_1,z_1')}^{d}p_{\min(z_1,z_1')}^{d-k}\\
 &\leq\underbrace{\sum_{\substack{z,z'\in Q\cap \llbracket 1;n\rrbracket^{d}\\ \mid I \mid=0}}\P(A_z)\P(A_{z'})}_{=:S_{1,n}}+\sum_{k=1}^d 2\binom{d-1}{k-1} \underbrace{\sum_{z_1=1}^n z_1^{2d-k-1}p_{z_1}^{2d-k}}_{=:S_{2,k,n}}\\
 &\;\;\;\;+ \sum_{k=1}^{d-1}2\binom{d-1}{k}\underbrace{\sum_{z_1=1}^n z_1^{d-1}p_{z_1}^{d-k}\sum_{z_1'=z_1+1}^n 
    z_1'^{d-k-1} p_{z_1'}^d}_{=:S_{3,k,n}}.
    \end{align*}
    We will now estimate each part of the right-hand term.
    First, we notice that $S_{1,n}\leq V_n$.
     Let us first estimate $S_{2,k,n}$ for $k\in\llbracket 1;d \rrbracket$. First we write:
     $$
S_{2,k,n}=\sum_{i=1}^{n}i^{2d-k-1}p_i^{2d-k}=\sum_{i=1}^{n}i^{d-1}p_i^{d}(ip_i)^{d-k}.
     $$
     Let $k\in\llbracket 0;d-1 \rrbracket$ and $i\in\N^{\ast}$ there exists a constant $C_k>0$ such that:
     \begin{align*}
         i(ip_i)^{k}&=i^{k+1}p_i^{k}\\
                    &\leq \left(\sup_{i_0\in\N}\frac{i_0^{k+1}}{\sum_{j=1}^{i_0}j^k}\right)\sum_{j=1}^{i}j^{k}p_j^{k}\text{ because }(p_i) \text{ is a decreasing sequence }\\
                    &=C_k\sum_{j=1}^{i} j^{\frac{k}{d}}j^{\frac{k(d-1)}{d}}p_j^k\\
                    &\leq C_k (\sum_{j=1}^{i} j^{\frac{k}{d-k}})^{\frac{d-k}{d}}(\sum_{j=1}^{i}j^{d-1}p_j^{d})^{\frac{k}{d}} \text{ according to Hölder's inequality}\\
                    &\leq C_k i^{(\frac{k}{d-k}+1)\frac{d-k}{d}}(\sum_{j=1}^{i}j^{d-1}p_j^{d})^{\frac{k}{d}}\\
                    &= C_k i(\sum_{j=1}^{i}j^{d-1}p_j^d)^\frac{k}{d}.
     \end{align*}
      We then have:
     \begin{equation*}\label{espinfbis}
     \forall k\in\llbracket 0;d-1 \rrbracket,\; (ip_i)^{k}\stackrel{i\to+\infty}{=}o(\sum_{j=1}^{i}j^{d-1}p_j^{d}).
     \end{equation*}
     It follows that we have:
     \begin{equation*}\label{espinf}
     \forall k\in\llbracket 1;d \rrbracket,\; (ip_i)^{d-k}\stackrel{i\to+\infty}{=}o(\sum_{j=1}^{i}j^{d-1}p_j^{d}).
     \end{equation*}
 Since $\sum_{j=1}^{i}j^{d-1}p_j^{d}\leq \sum_{j=1}^{n}j^{d-1}p_j^{d}$, we conclude that:
$$
\forall k\in \llbracket 1;d \rrbracket, S_{2,k,n}\stackrel{n\to+\infty}{=} o(V_n).
$$
Let us now estimate $S_{3,k,n}$. We see that :
$$
S_{3,k,n}=\sum_{i=1}^{n}i^{d-1}p_i^{d-k}\sum_{j=i+1}^{n}j^{d-k-1}p_{j}^{d}=\sum_{j=2}^{n}j^{d-k-1}p_{j}^{d}\sum_{i=1}^{j-1}i^{d-1}p_{i}^{d-k}.
$$
Now we have:

\begin{align*}
    \sum_{i=1}^{j-1}i^{d-1}p_{i}^{d-k}&=  \sum_{i=1}^{j-1} i^{\frac{(d-1)k}{d}}p_i^{d-k}i^{\frac{(d-1)(d-k)}{d}}\\
                        & \leq (\sum_{i=1}^{j-1} i^{\frac{d}{k}\frac{(d-1)k}{d}})^{\frac{k}{d}}( \sum_{i=1}^{j-1}p_i^{(d-k)\frac{d}{d-k}}i^{\frac{(d-1)(d-k)}{d}\frac{d}{d-k}})^{\frac{d-k}{d}}\text{ thanks to Hölder's inequality}\\
                        &\leq C_k j^{\frac{(d-1)k}{d}+\frac{k}{d}}(\sum_{i=1}^{j-1}i^{d-1}p_i^{d})^{\frac{d-k}{d}}\\
                        &\leq C_k j^k(\sum_{i=1}^{j-1}i^{d-1}p_i^{d})^{\frac{d-k}{d}}.
\end{align*}
Then by summing we finally have that:
$$
\forall k \in \llbracket 1;d \rrbracket, S_{3,k,n}\stackrel{n\to +\infty}{=}o(V_n).
$$
To conclude, since $S_{1,n}\leq V_n \text{\;and\;} S_{2,k,n}+S_{3,k,n}\stackrel{n\to+\infty}{=} o(V_n)$, we have:
$$
\limsup_{n\rightarrow+\infty}\frac{V_n}{S_{1,n}+S_{2,k,n}+S_{3,k,n}}\geq 1.
$$
Thanks to Lemma \ref{lem:BC}, Proposition \ref{prop:condition} is proven.
\end{proof}
\section{Brochette First-Passage Percolation with infinite passage times}\label{Section 5}
In this section, we allow our passage times to be infinite:
 $$
\P(\tau=+\infty)>0.
 $$
 The proof used here follows the same outline as \cite{Theret} by Cerf and Théret. They proved similar results in the context of classical first-passage percolation. Their idea is to introduce a regularized notion of passage times.
 For any $M>a$, we consider the set of edges whose passage time is strictly less than $M$. The cluster formed by these edges and their endpoints is denoted by $\mathcal{C}_M$. It is easy to see that $\mathcal{C}_M$ is infinite, connected, and composed of integer lines.
 Moreover, we can then define, for all $x,y\in\mathbb{Z}^{d}$, $\Tilde{T}_{M}(x,y)=T(\tilde{x}_{M},\tilde{y}_{M})$ where $\tilde{x}_{M}$ is the point in the cluster $\mathcal{C}_{M}$  closest to $x$ in $L^{1}$ norm with an arbitrary rule to break ties.
 In the same way, we define $\mathcal{C}_\ast$ the cluster $\mathcal{C}_M$ with $M=+\infty$.
 The proof of Theorem \ref{th:infstar} is intertwined with that of the following proposition, which shows that $\tilde{T}_M$ does not depend on $M$ asymptotically. 
\begin{prop}\label{prop:inftilde}
  For the Brochette first passage percolation, if $\P(\tau<+\infty)>0$, we have for all $M>a$:
$$
\forall x\in\mathbb{Z}^{d},\text{\;\;\;} \lim_{n\rightarrow+\infty} \frac{\tilde{T}_{M}(0,nx)}{n}=a\lVert x\rVert_1 \text{\;\;} a.s. \text{\;\;and\;in\;}L^{1}.
$$
\end{prop}
First to understand and then to prove this theorem we now define several objects.  If $M>a$, we define the following objects:
\begin{itemize}
\item $p_{M}:=\mathbb{P}(\tau\leq M).$
\item If $x,y\in\mathbb{Z}^{d}$, $D_{M}(x,y)=\inf\{\mid \gamma\mid: \gamma \text{\;path\;from\;} x \text{\;to\;} y \text{\;in\;} \mathcal{C}_M\} $ is the graph distance in $\mathcal{C}_M$, often called the chemical distance between $x$ and $y$ in $\mathcal{C}_M$ in the context of first-passage percolation.
\item $D(x,y)=\inf\{\mid \gamma\mid: \gamma \text{\;path\;from\;} x \text{\;to\;} y \text{\;in\;} \mathcal{C}_\ast\}$ is the chemical distance between $x$ and $y$ in $\mathcal{C}_\ast$.
\end{itemize}
To prove Proposition \ref{prop:inftilde}, we will once again use Kingman's subadditive theorem. The only assumption of Kingman's Theorem which is not immediate is the integrability of the process $(\tilde{T}_{M}(mx,nx))_{m<n}$. This is the purpose of our first lemma.
\begin{lem}\label{th:chimique}
    Consider Brochette first-passage percolation, then for all $M>a$ there exist $A_1(M),A_2(M)>0$ such that for all $x\in\mathbb{Z}^{d}$ and for all $\ell\geq (2d+1)\lVert x \rVert_1$
$$
\mathbb{P}(0\leftrightarrow x, D_M(0,x)\geq \ell)\leq A_1(M)\exp(-A_2(M) \ell).
$$
\end{lem}
\begin{proof}[Proof of Lemma \ref{th:chimique}]
 We assume that $x=(x_1,\dots,x_d)$ has all its coordinates positive (the other cases are treated in the same way). Let $\ell\geq(2d+1)\lVert x \rVert_1$. We consider the situation where only one edge emanating from $0$ and from $x$ are present, and they are parallel (this is the worst case; the proof is similar and sometimes shorter in other cases). Moreover to simplify the notations we assume that these edges are $\{0;e_1\}$ and $\{x;x+e_1\}$.
    Let us define $\underline{\ell}:=\lfloor\frac{\ell}{2d+1}\rfloor$. If $x$ and $y$ belong to the same integer line $\delta$, we denote  $x\underset{M}{\rightarrow}y$ if the passage time of $\delta$ is less than $M$, $i.e.$, if $x,y,\delta\in\mathcal{C}_M$.
    Let $1\leq c \leq \underline{\ell}$. We say that $(c,0\dots,0)$ is a good $0$-point if the event
    $$
(c,0,\dots,0)\underset{M}{\rightarrow}(c,\underline{\ell},0,\dots,0)\underset{M}{\rightarrow}(c,\underline{\ell},\underline{\ell},0,\dots,0)\underset{M}{\rightarrow}(c,\underline{\ell},\dots,\underline{\ell})
    $$
    occurs. Similarly, if $1\leq c' \leq \underline{\ell}$ we say that $(c',0,\dots,0)$ is a good $x$-point if the event:
    $$
(c',x_1,\dots,x_d)\underset{M}{\rightarrow}(c',\underline{\ell},x_3,\dots,x_d)\underset{M}{\rightarrow}(c,\underline{\ell},\underline{\ell},x_4,\dots,x_d)\underset{M}{\rightarrow}(c',\underline{\ell},\dots,\underline{\ell})
    $$
    occurs. The number $B$ of good $0$-points and the number $B'$ of good $x$-points follow a binomial distribution with parameters $\underline{\ell}$ and $p_M^{d-1}$. Lastly, if we denote $S$ the set of of integers $1\leq e \leq \underline{\ell}$ such that $(0,\underline{\ell},\dots,\underline{\ell},e)\underset{M}{\rightarrow}(\underline{\ell},\dots,\underline{\ell},e)$ then its cardinal $\tilde{B}$ follows a binomial distribution with parameters $\underline{\ell}$ and $p_M$. Moreover, if $B,B',\tilde{B}>0$, then there exists a path connecting $0$ to $x$ in $\mathcal{C}_M$ with a length at most $(2d+1)\underline{\ell}\leq \ell$.  We have illustrated the construction in the case of dimension $3$ in Figure \ref{Figure 5}.
\begin{figure}
    \centering
    \begin{subfigure}{0.4\textwidth}
        \centerline{\includegraphics[scale=0.25]{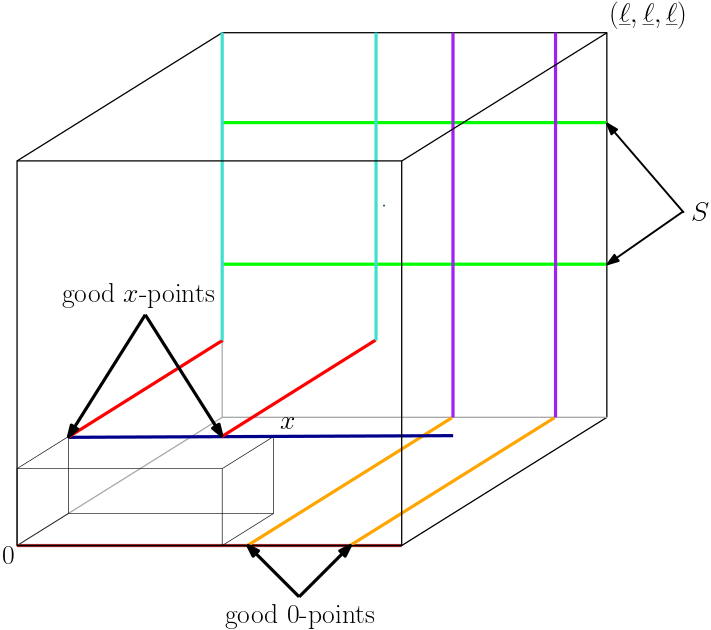}}~
        \caption{Illustration of Lemma \ref{th:chimique} in the case where the edges emanating from $0$ and $x$ are parallel.}
    \end{subfigure}
    \hspace{0.05\textwidth}
    \begin{subfigure}{0.4\textwidth}
        \centerline{\includegraphics[scale=0.25]{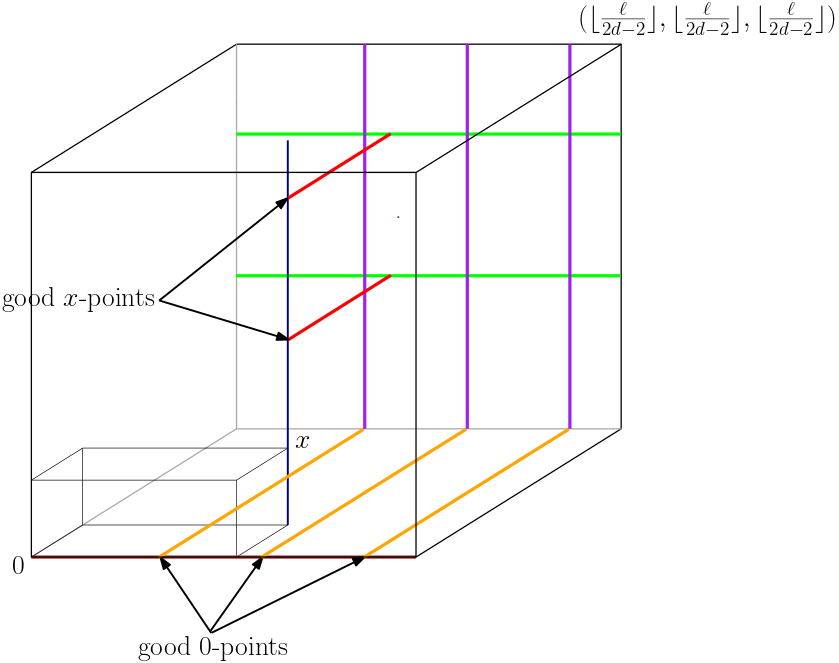}}~
        \caption{Illustration of Lemma \ref{th:chimique} in the case where the edges emanating from $0$ and $x$ are  not parallel.}
        \label{Figure 5 (b)}
    \end{subfigure}
    \caption{}
    \label{Figure 5}
\end{figure}
    We leave it to the readers to convince themselves that the construction is similar when the edges emanating from $0$ and $x$ are not parallel (see Figure \ref{Figure 5 (b)}).  
    Let us show that the non-existence of such a path has an exponentially small probability to occur. We have:
    \begin{align*}
    \P(0\leftrightarrow x,D_M(0,x)\geq\ell)&\leq \P(0\leftrightarrow x,D_M(0,x)\geq (2d+1)\underline{\ell})\\
    &\leq \P(B=0)+\P(B'=0)+\P(\tilde{B}=0)\\&\leq 3\P(\tilde{B}=0)\\&\leq 3\exp(\underline{\ell}\log(1-p_M))\\&
    \leq A_1(M)\exp(-A_2(M)\ell).
    \end{align*}
\end{proof}
\begin{prop}\label{prop:integ}
There exist three positive constants $C_1,C_2$ and $C_3$ which depend on $M$ such that for all $x\in\Z^d$ and for all $\ell\geq C_3\lVert x \rVert_1$
$$
 \P(\tilde{T}_M(0,x)>\ell)\leq C_1e^{-C_2\ell}.
$$
 \end{prop}
 \begin{proof}[Proof of Proposition \ref{prop:integ} ]
 \begin{equation*}
\begin{split}
\P(\tilde{T}_M(0,x)>\ell)\leq& \P(\tilde{0}\notin B_{1}(0,\frac{\epsilon \ell}{2}))+\P(\tilde{x}\notin B_{1}(x,\frac{\epsilon \ell}{2}))\\
+&\sum_{\substack{a\in B_{1}(0,\frac{\epsilon \ell}{2})\cap \Z^{d}\\b\in B_{1}(x,\frac{\epsilon \ell}{2})\cap\Z^{d}}}\P(\tilde{0}=a,\tilde{x}=b,D_{M}(a,b)\geq\frac{\ell}{M})\\
                            \leq&\P(\mathcal{C}_{M}\cap B_{1}(0,\frac{\epsilon \ell}{2})=\emptyset)+\P(\mathcal{C}_{M}\cap B_{1}(x,\frac{\epsilon \ell}{2})=\emptyset)\\
                            +&\sum_{\substack{a\in B_{1}(0,\frac{\epsilon \ell}{2})\cap \Z^{d}\\b\in B_{1}(x,\frac{\epsilon \ell}{2})\cap\Z^{d}}}\P(a \overset{M}{\longleftrightarrow}b,D_{M}(a,b)\geq\frac{\ell}{M}).
\end{split}
\end{equation*}
To bound the first two terms, we use the following rough estimate:
$$
\P(\mathcal{C}_{M}\cap B_{1}(0,r)=\emptyset)\leq (1-p_M)^r.
$$
Moreover, if $a\in B_{1}(0,\frac{\epsilon \ell}{2})\cap \Z^{d}$ and $b\in B_{1}(x,\frac{\epsilon \ell}{2})\cap\Z^{d}$ then $\lVert a-b \rVert_1\leq \lVert x \rVert_1+\epsilon \ell.$ Thus, for $\epsilon=\frac{1}{2(2d+1)M}$ and $\ell\geq 2M(2d+1)\lVert x \rVert_1$ we can use Lemma \ref{th:chimique} and we then have:
$$
\P(\tilde{T}_M(0,x)>\ell)\leq 2(1-p_M)^{\frac{\epsilon\ell}{2}} +\mid B_{1}(0,\frac{\epsilon \ell}{2})\cap\Z^{d} \mid^2 A_1\exp(-A_2\ell).
$$
Hence the result.
\end{proof}
\begin{proof}[Proof of Proposition \ref{prop:inftilde}] 
We first fix $M>a$.
All the assumptions of Kingman's Theorem are satisfied, and $\tilde{T}_{M}(0,x)$ is integrable thanks to Proposition \ref{prop:integ}. Then we know that:
\begin{equation}\label{convergenceinf}
\frac{\tilde{T}_{M}(0,nx)}{n}\rightarrow_{n\rightarrow+\infty}X_{M} \text{ a.s. and in }L^1,
\end{equation}
with $X_{M}$ a random variable. However, as in the case where passage times are finite almost surely, the lack of ergodicity implies that, for now, we do not know that the limiting random variable is almost surely equal to a constant. To see that we will study the convergence of a subsequence of $(\frac{\tilde{T}_{M}(0,nx)}{n})_{n\geq1}$ in the following lemma. 
\begin{lem}\label{lem:encadrement}
    For all $\epsilon>0$, $x\in \Z^d$ and $M>a$ , we have:
    $$
a\lVert x \lVert_1\leq X_{M}\leq M\lVert x \lVert_1 \; a.s..
    $$
 \end{lem}
 \begin{proof}[Proof of Lemma \ref{lem:encadrement}]
 Let $x=(x_1,\dots,x_d)\in\Z^d$. By symmetry we can assume that $x_1,\dots,x_d\geq0$.  
 By definition, there exists at least one integer line with a passage time less than $M$ passing through $\tilde{0}.$ Without loss of generality, assume that the line $\tilde{0}+\Z e_1$ has a passage time less than $M.$ Let $J=\{n\in\N, nx\in\mathcal{C}_M,\;D_M(\tilde{0},nx)\leq \lVert \tilde{0}-nx\rVert_1+2\sqrt{n}\}.$ Let us show that $J$ is infinite almost surely. If $nx\in\mathcal{C}_M$ is such that $nx+\Z e_1$ has a passage time less than $M$ we consider $K(n)$ to be the abscissa of the point in $x_1e_1+\N e_1$ closest to $x_1$ such that:
 $$
\tilde{0}+ K(n) e_1 \underset{M}{\rightarrow} (\tilde{0}+K(n),x_2,\tilde{0}_3,\dots,\tilde{0}_d)
\underset{M}{\rightarrow} \dots \underset{M}{\rightarrow} (\tilde{0}+K(n),x_2,\dots,x_d)\underset{M}{\rightarrow} (x_1,x_2,\dots,x_d)=x.
 $$ We then have that $D_M(\tilde{0},nx)\leq \lVert \tilde{0}-nx\rVert_1+2K(n)$. $K(n)$ follows a geometric distribution with parameter $p:=p_M^{d-1}$, so we have:
 $$
\P(K(n)>\sqrt{n})=(1-p)^{\lfloor \sqrt{n}\rfloor}.
 $$
Consequently by the Borel-Cantelli lemma $K(n)\leq\sqrt{n}$ infinitely often, almost surely. It follows that $J$ is infinite almost surely. Thus we have:
 \begin{align*}
     \lim_{n\to+\infty}\frac{\tilde{T}_M(0,nx)}{n}&=\lim_{\substack{n\to+\infty\\ n\in J}}\frac{\tilde{T}_M(0,nx)}{n}\\
        &\leq M \liminf_{\substack{n\to+\infty\\ n\in J}}\frac{D_M(\tilde{0},nx)}{n}\\
        &\leq M \liminf_{\substack{n\to+\infty\\ n\in J}}(\frac{\lVert \tilde{0}-0\rVert_1}{n}+\frac{\lVert 0-nx\rVert_1}{n})\\
        &\leq M\lVert x \rVert_1.
 \end{align*}
 Similarly, we have:
 \begin{align*}
     \lim_{n\to+\infty}\frac{\tilde{T}_M(0,nx)}{n}&\geq \lim_{\substack{n\to+\infty\\ n\in J}}a\frac{D_M(\tilde{0},nx)}{n}\\
     &\geq a \lim_{\substack{n\to+\infty\\ n\in J}}\frac{\lVert \tilde{0}-nx \rVert_1}{n}=a \lVert x \rVert_1.
 \end{align*}
 Thus, according to \eqref{convergenceinf}, if $M>a$ we have:
 \begin{equation}\label{limite}
a\lVert x \rVert_1\leq X_{M}\leq M\lVert x \rVert_1 a.s..
\end{equation}
\end{proof}
So, we need to show that, in fact, $X_{M}$ does not depend on $M$. To do this, let us fix $x\in\Z^d$. Recall that $T^{\ast}(x,y)=T(x^{\ast},y^{\ast})$ where $x^{\ast}$ is the closest vertex in $L^{1}$ norm to $x$ in $\mathcal{C}_{\ast}$ with an arbitrary rule to break ties. According to the triangular inequality on $T$, 
$$
\mid T^{\ast}(0,nx)-\tilde{T}_{M}(0,nx)\mid\leq T(0^\ast,\tilde{0})+T(nx^\ast,\tilde{nx}).
$$
Since $T(0^\ast,\tilde{0})$ and $T(nx^\ast,\tilde{nx})$ have same distribution :
$$
\P(\mid T^{\ast}(0,nx)-\tilde{T}_{M}(0,nx)\mid\geq \epsilon n)\leq \P(T(0^\ast,\tilde{0})+T(nx^\ast,\tilde{nx})\geq \epsilon n)\leq 2\P(T(0^\ast,\tilde{0})\geq \frac{\epsilon n}{2}).
$$
Since $0^\ast$ and $\tilde{0}$ are both in $\mathcal{C}_\ast$, $T(0^\ast,\tilde{0})<+\infty$ and the above inequality implies:
\begin{equation}\label{unilim}
\forall x\in \Z^d, \; \frac{T^{\ast}(0,nx)}{n}-\frac{\tilde{T}_M(0,nx)}{n}\overset{\P}{\rightarrow}0.
\end{equation}
We can conclude that $\frac{T^{\ast}(0,nx)}{n}$ and $\frac{\tilde{T}_M(0,nx)}{n}$ have the same limit which can not depend on $M$ since $\frac{T^{\ast}(0,nx)}{n}$ does not depend on $M$.
By letting $M$ tend to $a$ in \eqref{limite}, we deduce that: 
\begin{equation}\label{convas}
\frac{\tilde{T}_M (0,nx)}{n} \rightarrow_{n \rightarrow + \infty} a \lVert x \rVert_1 \text{\;\;\;almost \; surely},
\end{equation}
which concludes the proof of Theorem \ref{prop:inftilde}.
\end{proof}
\begin{proof}[Proof of Theorem \ref{th:infstar}]
Combining \eqref{unilim} and \eqref{convas} we obtain:
$$
\frac{T^{\ast}(0,nx)}{n} \overset{\P}{\rightarrow}_{n \rightarrow + \infty} a \lVert x \rVert_1.
$$
\end{proof}

\appendix
\section{Appendix}
The asymptotic of $M_n$ is studied in numerous references, but we have not found any reference regarding the study of $\E[M_n]$.
For the sake of self-containedness
we provide an estimate of $\E[M_n]$.
\begin{lem}\label{lem:tpsborne}
 Assume that  the passage times are bounded. Let $(\tau_i)_{i\in\N}$ be i.i.d. random variables with cumulative distribution function $F$ such that:
$$
F(t)\sim_{t\rightarrow 0^{+}}Ct^{\beta},\text{\;\;\;for\;a\;constant\;} C>0 \text{ and }\beta>0.
$$
Then, if $M_n=\min(\tau_1,\dots,\tau_n)$ we have:
$$
\mathbb{E}[M_{n}]\sim_{n\rightarrow +\infty}\frac{1}{n^{\frac{1}{\beta}}}\frac{1}{\beta C^{\frac{1}{\beta}}}\Gamma(\frac{1}{\beta}).
$$
\end{lem}
\begin{proof}[Proof of Lemma \ref{lem:tpsborne}] 
Let $K>0$  be such that $\tau\leq K $ almost surely.\\
Let us denote $X_{n}:=M_{n}n^\frac{1}{\beta}$. We have:\\
$$
    \mathbb{E}[X_{n}]=\int_{0}^{+\infty}\mathbb{P}(X_{n}\geq t)dt
                     =\int_{0}^{Kn^{\frac{1}{\beta}}}(1-F(\frac{t}{n^{\frac{1}{\beta}}}))^{n}dt.
$$
Now, for a fixed $t$, we have:
$$
    (1-F(\frac{t}{n^{\frac{1}{\beta}}}))^{n}=(1-C\frac{t^{\beta}}{n}+ o(\frac{1}{n}))^{n}
                    \longrightarrow_{n\rightarrow+\infty}\exp(-Ct^{\beta}).
$$
Let $C'>0$ such that  $F(u)\geq C'u^\beta$ for $u\in[0;K]$. Using that $(1-\frac{x}{n})^n\leq e^{-x}$, we have that for all $t\geq0$, $(1-F(\frac{t}{n^{\frac{1}{\beta}}}))^n\leq \exp(-C't^\beta)$.
According to the Dominated Convergence Theorem and with the changes of variables $u=Ct^{\beta}$ we have:
$$
    \mathbb{E}[X_{n}]\rightarrow_{n\rightarrow+\infty}\int_{0}^{+\infty}\exp(-Ct^{\beta})dt=\Gamma(\frac{1}{\beta})\times \frac{1}{\beta C^{\frac{1}{\beta}}}.
$$
So we finally have:
$$
\mathbb{E}[M_{n}]\sim_{n\rightarrow+\infty}\frac{1}{n^{\frac{1}{\beta}}}\frac{1}{\beta C^{\frac{1}{\beta}}}\Gamma(\frac{1}{\beta}).
$$
 \end{proof}
\section*{Acknowledgments}
I would like to warmly thank my two PhD advisors, Anne-Laure Basdevant and Lucas Gerin, first for suggesting this topic and for their support throughout the writing of this article. This work was partially supported by ANR LOUCCOUM.

\bibliographystyle{alpha}
\bibliography{bibliographie}

\begin{thebibliography}{DCHKS18}

\bibitem[ADH17]{Survey}
Antonio Auffinger, Michael Damron, and Jack Hanson.
\newblock {\em 50 years of first-passage percolation}, volume~68.
\newblock American Mathematical Soc., 2017.

\bibitem[Boi90]{fppergo}
Daniel Boivin.
\newblock First passage percolation: the stationary case.
\newblock {\em Probability theory and related fields}, 86(4):491--499, 1990.

\bibitem[CCD86]{chayes}
Jennifer~Tour Chayes, Lincoln Chayes, and Richard Durrett.
\newblock Critical behavior of the two-dimensional first passage time.
\newblock {\em Journal of statistical physics}, 45:933--951, 1986.

\bibitem[CD81]{Cox}
John~Theodore Cox and Richard Durrett.
\newblock Some limit theorems for percolation processes with necessary and
  sufficient conditions.
\newblock {\em The Annals of Probability}, pages 583--603, 1981.

\bibitem[CT16]{Theret}
Rapha{\"e}l Cerf and Marie Th{\'e}ret.
\newblock Weak shape theorem in first passage percolation with infinite passage
  times.
\newblock In {\em Annales de l’Institut Henri Poincar{\'e}-Probabilit{\'e}s
  et Statistiques}, volume~52, pages 1351--1381, 2016.

\bibitem[DCHKS18]{Brochette}
Hugo Duminil-Copin, Marcelo~Richard Hilario, Gady Kozma, and Vladas
  Sidoravicius.
\newblock Brochette percolation.
\newblock {\em Israel Journal of Mathematics}, 225:479--501, 2018.

\bibitem[DLW17]{damron}
Michael Damron, Wai-Kit Lam, and Xuan Wang.
\newblock Asymptotics for 2d critical first passage percolation.
\newblock {\em The Annals of Probability}, 45(5):2941--2970, 2017.

\bibitem[ER59]{Lem:BC}
Paul Erd{\H{o}}s and Alfréd R{\'e}nyi.
\newblock On {C}antor's series with convergent $\sum \frac{1}{q_n}$.
\newblock {\em Ann. Univ. Sci. Budapest. E{\"o}tv{\"o}s. Sect. Math.},
  2:93--109, 1959.

\bibitem[GM04]{garetmarchand}
Olivier Garet and R{\'e}gine Marchand.
\newblock Asymptotic shape for the chemical distance and first-passage
  percolation on the infinite bernoulli cluster.
\newblock {\em ESAIM: Probability and Statistics}, 8:169--199, 2004.

\bibitem[HW65]{Hammersley}
John~Michael Hammersley and Dominic Welsh.
\newblock First-passage percolation, subadditive processes, stochastic
  networks, and generalized renewal theory.
\newblock In {\em Bernoulli 1713, Bayes 1763, Laplace 1813: Anniversary Volume.
  Proceedings of an International Research Seminar Statistical Laboratory
  University of California, Berkeley 1963}, pages 61--110. Springer, 1965.

\bibitem[Kes86]{Kesten}
Harry Kesten.
\newblock Aspect of first passage percolation.
\newblock {\em Lecture Notes in Math.}, 1180:126--263, 1986.

\bibitem[Kin73]{kingmanpap}
John Frank~Charles Kingman.
\newblock Subadditive ergodic theory.
\newblock {\em The Annals of Probability}, pages 883--899, 1973.

\bibitem[Lig85]{Kingman}
Thomas~Milton Liggett.
\newblock An improved subadditive ergodic theorem.
\newblock {\em The Annals of Probability}, 13(4):1279--1285, 1985.

\bibitem[Zha95]{zhang}
Yu~Zhang.
\newblock Supercritical behaviors in first-passage percolation.
\newblock {\em Stochastic processes and their applications}, 59(2):251--266,
  1995.

\bibitem[Zha99]{zhang2}
Yu~Zhang.
\newblock Double behavior of critical first-passage percolation.
\newblock {\em Perplexing Problems in Probability: Festschrift in Honor of
  Harry Kesten}, pages 143--158, 1999.

\end{thebibliography}
\end{document}